\documentclass[12pt,a4paper]{article}  
\usepackage[utf8]{inputenc}
\usepackage[english]{babel}

\usepackage[bookmarks=false]{hyperref} 
\hypersetup{
	colorlinks=true,
	linkcolor=teal,
	urlcolor=cyan,
	citecolor=violet,
}

\usepackage[theorems]{default}

\usepackage[left=2cm,right=2cm,top=2.5cm,bottom=2cm,headheight=15pt]{geometry}

\theoremstyle{plain}
\newtheorem*{theorem Etemadi}{Etemadi's Theorem (\cite{Etemadi})}
\newtheorem*{claim}{Claim}

\theoremstyle{remark}
\newtheorem*{remark*}{Remark}

\theoremstyle{definition}

\newtheorem{example**}{Example}
\newtheorem*{notation}{Notation}

\newcommand*{\Folge}[2]{(#1_#2)_{#2 \in \mathbb{N}}} 

\newcommand*{\Erwartung}[1]{\mathsf{E} \hspace{0.03cm} {#1}} 
\newcommand*{\Varianz}[1]{\mathsf{V} \hspace{0.03cm} {#1}} 
\newcommand*{\Prob}[1]{\mathsf{P} \hspace{0.04cm} {#1}} 
\newcommand*{\define}{\coloneqq}

\DeclareMathOperator*{\essinf}{ess\,inf}

\author{Maximilian Janisch}

\newcommand{\customfootnotetext}[2]{{
\renewcommand{\thefootnote}{#1}
\footnotetext[0]{#2}}}

\begin{document}
\begin{center}  
	{\Large\itshape Maximilian Janisch\textsuperscript{*}\par}
	\customfootnotetext{*}{The paper was initially written for a student seminar on probability theory under supervision of Prof. Jean Bertoin, University of Zürich in early 2019. At the time I was 15 years old. My email address is \href{mailto:mail@maximilianjanisch.com}{mail@maximilianjanisch.com}. Currently I am a graduate student at the Institute of Mathematics, University of Zürich (190, Winterthurerstrasse, 8057 Zürich, Switzerland).
}
	\vspace{0.6cm}
	{\LARGE\bfseries Kolmogoroff's Strong Law of Large Numbers holds for pairwise uncorrelated random variables\par}
	\vspace{0.6cm}
	{\large Original version written on May 8, 2020; Text last updated January 31, 2022;\\Journal information updated on October 27, 2021\par}
	\thispagestyle{empty}
\end{center}
\vspace{0.6cm}
\begin{abstract}
	Using the approach of N. Etemadi for the Strong Law of Large Numbers (\textit{SLLN}) in \cite{Etemadi} and its elaboration in \cite{Csoergo}, I give weaker conditions under which the SLLN still holds, namely for pairwise uncorrelated (and also for \enquote{quasi uncorrelated}) random variables. I am focusing in particular on random variables which are not identically distributed. My approach here leads to another simple proof of the classical SLLN.
	
	\smallskip
	\textsc{Published in (English edition):} \emph{Theory of Probability and its Applications}, 2021, Volume 66, Issue 2, Pages 263–275. \\ \url{https://epubs.siam.org/doi/10.1137/S0040585X97T990381}.\\
	\textsc{Doi:} \url{https://doi.org/10.1137/S0040585X97T990381}.
	
	\smallskip
	\textsc{Published in (Russian edition):} \emph{Teoriya Veroyatnostei i ee Primeneniya} (this is the Russian edition of \emph{Theory of Probability and its Applications}), 2021, Volume 66, Issue 2, Pages 327–341. (Submitted November 23, 2020; Accepted November 25, 2020).\\ \url{http://mi.mathnet.ru/tvp5459}. \\
	\textsc{Doi:} \url{https://doi.org/10.4213/tvp5459}.
\end{abstract}

\section{Introduction and results}
In publication \cite{Etemadi} from 1981, N. Etemadi weakened the requirements of Kolmogoroff’s first SLLN for identically distributed random variables and provided an elementary proof of a more general SLLN. I will first state the main Theorem of \cite{Etemadi} after introducing some notations that I will use.

\begin{notation}
	Throughout the paper, $(\Omega, \mathcal{A}, \mathsf{P})$ is a probability space. 
	All random variables (usually denoted by $X$ or $X_n$) are assumed to be measurable functions from $\Omega$ to the set of real numbers $\mathbb R$. If $X$ is a random variable, I write $\mathsf EX$ for its expected value and $\mathsf VX$ for its variance.
\end{notation}

\begin{theorem Etemadi}
	Let $\Folge{X}{n}$ be a sequence of \emph{pairwise} independent, identically distributed random variables with $\Erwartung{\abs*{X_1}}<\infty$. Then 
	\begin{equation}\label{eq:Etemadi conclusion}
		\lim_{n \to \infty} \frac{X_1+X_2+\dots+X_n}{n} =  \Erwartung{X_1} \quad \text{almost surely (a.s.)}.
	\end{equation}
\end{theorem Etemadi}

\begin{remark}
	Following Remark 2 from chapter 4.3 in Shiryaev's book \cite{Shiryaev}, we notice that if $\mathsf E X_1=\infty$, then \eqref{eq:Etemadi conclusion} still holds.
\end{remark}

The main improvement of Etemadi's Theorem compared to Kolmogoroff's first SLLN is the replacement of \emph{independence} by \emph{pairwise independence}. The condition \enquote{identically distributed} still remains. There have been several extentions of this Theorem. For example, Chandra and Goswami gave a sufficient SLLN condition for non-identically distributed random variables. For convenience, I will from now on always use the notation $S_n\define X_1+\dots+X_n$.

\begin{theorem}[Chandra, Goswami \cite{Chandra-Goswami}]
	Let $(X_n)_{n\in\N}$ be a sequence of pairwise independent random variables such that
	\begin{equation*}
		\int_0^\infty \sup_{n\in\N} \mathsf P(\abs{X_n}>t)\,\mathrm dt<\infty.
	\end{equation*}
	Then 
	\begin{equation*}
		\lim_{n \to \infty} \frac{S_n-\mathsf E S_n}{n} = 0 \quad \text{a.s.}
	\end{equation*}
\end{theorem}

However, Etemadi's arguments enable an elegant proof for large classes of non-identically distributed random variables.
First, let me introduce a few useful notions:


\begin{definition}
	A sequence $\Folge{X}{n}$ of (arbitrarily distributed) random variables is said to 
	\begin{enumerate}[i.]
		\item \emph{Satisfy the Kolmogoroff condition} if and only if each one of the $X_n's, n\in\N$ has finite variance $\Varianz X_n$ and
		\begin{equation}\label{eq:KolmogoroffBedingung}
			\sum_{n\in\mathbb N} \frac{\Varianz X_n}{n^2} < \infty.
		\end{equation}
		\item \emph{Be quasi uncorrelated} if and only if each one of the $X_n's, n\in\N$ has finite variance and there exists a positive constant $c$ such that $\Varianz S_n \le  c\sum_{k=1}^n \Varianz X_k$ for all $n\in\mathbb N$.
		\item \emph{Satisfy the SLLN} if and only if
		\begin{equation*}
			\lim_{n\to\infty} \frac{S_n-\Erwartung S_n}n = 0 \quad \text{a.s.}
		\end{equation*}
		Notice that if $a \define \lim_{n\to\infty} \Erwartung X_n$ exists, this is (by Cesàro summation) equivalent to the relation $\lim_{n\to\infty} \frac{1}{n} S_n = a$ a.s.
	\end{enumerate}
\end{definition}

\begin{theorem}[SLLN under the Kolmogoroff condition \eqref{eq:KolmogoroffBedingung}] \label{Thm:KolmogoroffBedingung}
	Let $\Folge{X}{n}$ be a sequence of non-negative, quasi uncorrelated random variables satisfying the  Kolmogoroff condition \eqref{eq:KolmogoroffBedingung} such that $A\define\sup_{n\in\mathbb N} \frac{\Erwartung{S_n}}n < \infty$.
	Then the sequence $(X_n)_{n\in\mathbb N}$ satisfies the SLLN.
\end{theorem}

\begin{remark}
	After having finished the proof of this Theorem and the proof of Theorem \ref{Thm:PaarweiseUnabh}, I became aware of the paper by Korchevsky \cite{Korchevsky}, where the above Theorem \ref{Thm:KolmogoroffBedingung} is formulated as Lemma 1. Korchevsky refers to the work of Chandra-Goswami \cite{Chandra-Goswami}. For completeness, I will give a detailed proof of Theorem \ref{Thm:KolmogoroffBedingung}, whereas the proof in \cite{Chandra-Goswami} is only briefly outlined.
\end{remark}

\begin{remark}\label{rmk:Arbitrary sequence}
	The condition $\sup_{n\in\N}\frac{\mathsf E S_n}n<\infty$ can be replaced by $\sup_{n\in\N}\frac{\mathsf E S_n}{b_n}<\infty$ for any unbounded, non-decreasing sequence of positive real numbers $b_n$. In that case, if the condition $\sum_{n=1}^\infty \frac{\mathsf V X_n}{n^2}<\infty$ is replaced by $\sum_{n=1}^\infty \frac{\mathsf V X_n}{b_n^2}<\infty$, then the sequence $(X_n)_{n\in\N}$ will satisfy $\lim_{n\to\infty}\frac{S_n-\mathsf ES_n}{b_n}=0$ a.s. During the proof I will make remarks that show how the proof has to be adapted in order to work for such $b_n$.
\end{remark}

\begin{remark}
	Of course Theorem \ref{Thm:KolmogoroffBedingung} is also true if \emph{non-negative} is replaced by \emph{a.s. non-negative}. For my purposes here, these two notions can be seen as synonymous.
\end{remark}

\begin{remark}\label{paarweise folgt quasi}
	If the $X_n$'s are pairwise uncorrelated, we have the equality $\Varianz S_n = \sum_{k=1}^n \Varianz X_k$ (Bienaymé) and thus the $X_n$ are, in particular, quasi uncorrelated.
\end{remark}

\begin{remark}
	If $X$ is a random variable, I use the notation $\essinf X$ for the essential infimum of $X$, i.e. the largest number $r$ such that $X\geq r$ a.s. Applying Theorem \ref{Thm:KolmogoroffBedingung} to the transformation $\tilde X_n \define X_n - \essinf X_n$ shows that the condition $\sup_{n\in\mathbb N} \frac{\Erwartung{S_n}}n < \infty$ can be replaced by the weaker condition 
	\begin{equation}\label{eq:SupCondition}
		\sup_{n\in\mathbb N}\frac{\Erwartung{S_n}-\sum_{k=1}^n \essinf X_k}n < \infty.
	\end{equation}
	Since for any numerical sequence $a_n\geq0$, $\sup_{n\in\N} a_n=\infty$ implies 
	\begin{equation*}
		\limsup_{n\to\infty} a_n=\lim_{n\to\infty}\sup_{k\geq n}a_k=\infty,
	\end{equation*}
	condition \eqref{eq:SupCondition} is equivalent to
	\begin{equation*}
		\limsup_{n\to\infty}\frac{\Erwartung{S_n}-\sum_{k=1}^n \essinf X_k}n < \infty.
	\end{equation*}
\end{remark}

\begin{remark}
	My proof of Theorem \ref{Thm:KolmogoroffBedingung} relies heavily on the proof by Csörgo et al. \cite{Csoergo}.
	In their paper, it is also shown that the condition of non-negativity cannot be removed. More precisely, in their Theorem 4, the authors construct a sequence of pairwise uncorrelated random variables $X_n$ satisfying the Kolmogoroff condition \eqref{eq:KolmogoroffBedingung} and the condition
	\begin{equation*}
		\lim_{n\to\infty}\frac{\sum_{k=1}^n \Erwartung{}\abs{X_k-\Erwartung X_k}}n=0
	\end{equation*}
	and
	\begin{equation*}
		\limsup_{n\to\infty}\frac{\abs{S_n-\Erwartung S_n}}n=\infty.
	\end{equation*}
	If we require the $X_n$ to be pairwise independent instead of just pairwise uncorrelated, the non-negativity assumption can be dropped:
\end{remark}

\begin{corollary}\label{Cor:PaarweiseUnabhKolmogoroff}
	Let $\Folge{X}{n}$ be a sequence of pairwise independent random variables satisfying the  Kolmogoroff condition \eqref{eq:KolmogoroffBedingung} and suppose further that 
	\begin{equation}\label{eq:AuxilliaryCondition}
		\limsup_{n\to\infty}\frac{\sum_{k=1}^n \Erwartung{}\abs{X_k-\Erwartung X_k}}n < \infty.
	\end{equation}
	Then the sequence $\Folge Xn$ satisfies the SLLN.
\end{corollary}

\begin{remark}\label{NichtNegativitaet}
	Thanks to linearity of $\mathsf E$, the non-negativity condition in Theorem \ref{Thm:KolmogoroffBedingung} can be replaced by a.s. uniform boundedness of each $X_n$ (or even of $X_n-\Erwartung X_n$) from below or above. For example, if each $X_n$ is bounded from below by the real number $r$, then Theorem \ref{Thm:KolmogoroffBedingung} can be applied to $\tilde{X}_n\define X_n-r$. In fact, one can always try to use the transformation $\tilde{X}_n\define X_n-\essinf_{n\in\N} X_n$ and check if that transformation satiesfies all conditions of Theorem \ref{Thm:KolmogoroffBedingung}.
\end{remark}

\begin{remark}
	Kolmogoroff's classical SLLN is very similar to Corollary \ref{Cor:PaarweiseUnabhKolmogoroff}, except that pairwise independence is replaced by independence but the auxilliary condition \eqref{eq:AuxilliaryCondition} is dropped. However, as is proven in \cite{Csoergo}, Theorem 3, there is a sequence of pairwise independent random variables $\Folge Xn$ satisfying the Kolmogoroff condition \eqref{eq:KolmogoroffBedingung} such that the SLLN is not satisfied. Hence, an auxilliary condition like \eqref{eq:AuxilliaryCondition} in Corollary \ref{Cor:PaarweiseUnabhKolmogoroff} is necessary.
\end{remark}

\medskip
We can try to relax the Kolmogoroff condition \eqref{eq:KolmogoroffBedingung} for the $X_n$ using \emph{truncation arguments}. For this, however, the condition that all $X_n$'s are pairwise uncorrelated has to be replaced by a condition of pairwise independence. The reason for this is that the truncations of the $X_n$'s are still pairwise independent if the $X_n$'s are pairwise independent, while the truncations need not be pairwise uncorrelated if the $X_n$'s are only pairwise uncorrelated.

\medskip
More precisely, the truncations of $X_n$ that I will look at are $Y_n\define X_n\cdot 1_{\set{X_n\le n}}$, where $1_{\set{X_n\le n}}$ is the indicator function of the event $\set{X_n\le n}$. 
\smallskip

\begin{remark}\label{rem:Example which doesn't satisfy WLLN}
	It should be noted that the truncated variables $Y_n, n\in\N$, do not have to satisfy the Kolmogoroff condition \eqref{eq:KolmogoroffBedingung}! It depends on the $X_n$'s. For example, if the $X_n$'s are random variables with $\Prob(X_n = n)=\Prob(X_n=0)=\frac12$, then $\Varianz(Y_n)=\Varianz(X_n)=\frac{n^2}4$, which is too large for the Kolmogoroff condition \eqref{eq:KolmogoroffBedingung} to hold. In this case, however, it is not surprising, because the sequence $(X_n)_{n\in\N}$ also does not satisfy the SLLN. Interestingly, the sequence $(X_n)_{n\in\N}$ does not even satisfy the Weak Law of Large Numbers (WLLN). Indeed, let $\tilde X_n = X_n - \Erwartung X_n$. Then $S_n-\Erwartung S_n = \sum_{k=1}^n \tilde X_k$ and $\mathsf P\left(\tilde X_n = -\frac n2\right)=\mathsf P\left(\tilde X_n = \frac n2\right)=\frac12$, so that for $n\geq 2$,
	\begin{equation*}
		\Prob{} \left(S_n-\Erwartung S_n\geq \frac{n}2\right) = \Prob{}\left(\sum_{k=1}^n \tilde X_k\geq\frac n2\right)=\frac{\Prob{} \left(\sum_{k=1}^{n-1}  \tilde X_k \geq0\right)+\Prob{} \left(\sum_{k=1}^{n-1} \tilde X_k \geq n\right)}2\geq\frac14,
	\end{equation*}
	where the last inequality is true by the symmetry of $\tilde X_k$, i.e. $\Prob{} \left(\sum_{k=1}^{n-1}  \tilde X_k \geq0\right)\geq\frac12$.
\end{remark}

\begin{remark}
	Since the SLLN implies the WLLN, the sequence of random variables from Remark \ref{rem:Example which doesn't satisfy WLLN} satisfy neither the SLLN nor the WLLN. For completeness, I refer to examples 15.3 and 15.4 in \cite{Stoyanov} for sequences of random variables that satisfy the WLLN but not the SLLN.
\end{remark}

\begin{theorem}[SLLN for random variables with the Kolmogoroff condition for the truncations]\label{Thm:PaarweiseUnabh}
	Let $\Folge{X}{n}$ be a sequence of non-negative, pairwise independent random variables such that the truncated sequence $(Y_n)_{n\in\N}$, where $Y_n\define X_n \cdot 1_{\set{X_n\le n}}$, satisfies the Kolmogoroff condition \eqref{eq:KolmogoroffBedingung} as well as \eqref{eq:AuxilliaryCondition} and such that $X_n-Y_n$ goes to $0$ in $L^1$-sense and a.s. as $n\to\infty$. Then the sequence $(X_n)_{n\in\N}$ satisfies the SLLN.
\end{theorem}

\begin{remark}
	Remark \ref{NichtNegativitaet} also applies to Theorem \ref{Thm:PaarweiseUnabh}.
\end{remark}

\smallskip
Of course, we can slightly modify the assumptions to avoid non-negativity. The following two standard notations, $X_n^+\define \max\{X_n,0\}$ and $X_n^-\define-\min\{X_n,0\}$, will be used.

\begin{corollary}\label{Cor:ZuTheorem2}
	Let $\Folge{X}{n}$ be a sequence of random variables such that both sequences $\Folge{X^+}{n}$ and $\Folge{X^-}n$ satisfy the conditions of Theorem \ref{Thm:PaarweiseUnabh}. Then the sequence $(X_n)_{n\in\N}$ satisfies the SLLN.
\end{corollary}

\smallskip
Finally, I want to prove a (well-known) statement which shows that Etemadi's result follows from Corollary \ref{Cor:ZuTheorem2} since identically distributed random variables \enquote{behave nicely}:

\begin{theorem}\label{Prop:Identical Distribution}
	Let $\Folge Xn$ be a sequence of non-negative, identically distributed, integrable random variables. Then the truncated sequence $(Y_n)_{n\in\N}$, defined by $Y_n\define X_n\cdot1_{\set{X_n\le n}}$, satisfies the Kolmogoroff condition, and, as $n\to\infty$, $X_n- Y_n\to0$ both in $L^1$-sense and a.s.
\end{theorem}

If we are given any identically distributed, pairwise independent random variables $X_1,X_2,\dots$ with $\Erwartung{}\abs{X_1}<\infty$, we can apply Theorem \ref{Prop:Identical Distribution} to each one of the sequences $(X_n^+)_{n\in\N}$ and $(X_n^-)_{n\in\N}$ to see that all conditions of Corollary \ref{Cor:ZuTheorem2} are satisfied. As a consequence, we obtain the conclusion of Etemadi's Theorem.

\section{Examples}
In order to illustrate my results I will now give two examples.

\begin{example**}\label{example:Cosine}
Let $X$ be a (continuous) uniformly distributed random variable over the interval $[-1,1]$. Define the sequence 
\begin{equation}\label{eq:Definition of the X_n in first example}
	X_n\define\cos(2\pi n X) \text{ for $n\in\N$.}
\end{equation}
\begin{claim}
	The following three statements are true:
	\begin{enumerate}[1)]
		\item\label{enum:Uncorrelated} The sequence $(X_n)_{n\in\N}$ is pairwise uncorrelated.
		\item\label{enum:Dependent} Each pair $(X_i, X_j)$ for $i,j\in\N$ is not independent.
		\item\label{enum:SLLN} The sequence $(X_n)_{n\in\N}$ satisfies the SLLN.
	\end{enumerate}
\end{claim}
\begin{remark}
	As will be seen in the proof, my Theorem \ref{Thm:KolmogoroffBedingung} can be used to prove that $(X_n)_{n\in\N}$ satisfies the SLLN. Notice that, because of statement \ref{enum:Dependent}, the original Theorem in \cite{Csoergo} cannot be applied.
\end{remark}
\begin{proof}[Proof of the statements]
	Let me start with the most important statement, \ref{enum:SLLN}. Consider $\tilde{X_n}\define1+X_n$. Then (see \eqref{eq:Definition of the X_n in first example}), $\mathsf E(\tilde{X_n})=1+\mathsf E(X_n)=1$. Hence $\sup_{n\in\mathbb N}\frac{\mathsf E(\tilde{X_1}+\dots+\tilde{X_n})}n=1<\infty$. Additionally, all of the random variables $\tilde{X_n}$ are in the interval $[0,2]$, which implies that the Kolmogoroff condition is satisfied. Statement \ref{enum:Uncorrelated} guarantees that Theorem \ref{Thm:KolmogoroffBedingung} can be applied to conclude that the sequence $(X_n)_{n\in\N}$ satisfies the SLLN.
	
	\smallskip
	To prove statement \ref{enum:Uncorrelated}, we need to show that $\mathsf E(X_i X_j) =\mathsf E(X_i)\mathsf E(X_j)$ for all $i\neq j$.
	This is true by the following two simple computations:
	\begin{align*}\label{eq:Expected value for Cosine example}
		\mathsf E(X_i) = \int_{-1}^1 \cos(2\pi i x)\,\mathrm dx =\ &0 = \mathsf E(X_j),
		\\
		\mathsf E(X_i X_j)=\int_{-1}^1 \cos(2\pi i x)\cos(2\pi j x)\,\mathrm dx =\ &0.
	\end{align*}
	The last equality follows from the elementary trigonometric identity
	\begin{equation*}
		\cos(2\pi i x)\cos(2\pi j x)=\frac{\cos (2 \pi  (i-j) x)+\cos (2 \pi  (i+j) x)}{2}.
	\end{equation*}
	
	To prove statement \ref{enum:Dependent}, let $\varepsilon>0$ and suppose that $X_i>1-\varepsilon$ for some $i\in\mathbb N$. This means that $\cos(2\pi i X)>1-\varepsilon$. Consider the $\arccos$ \enquote{branch}, defined by 
	\begin{align}
		\arccos_i:[-1, 1]&\to[2 i \pi, (2 i+1) \pi],\\ \cos(x)&\mapsto x.
	\end{align}
	This function is continuous and thus there exist $\delta(\varepsilon)>0$ such that 
	\begin{equation*}
		2\pi i X\in]2\pi k-\delta(\varepsilon),2\pi k+\delta(\varepsilon)[
	\end{equation*}
	for some $k\in\mathbb Z$ and such that $\lim_{\varepsilon\to0}\delta(\varepsilon)=0$.
	This means that 
	\begin{equation*}
	X\in\bigcup_{k=-i}^i\left]\frac ki-\frac{\delta(\varepsilon)}i, \frac ki+\frac{\delta(\varepsilon)}i\right[\overset{\text{Def.}}=D.
	\end{equation*}
	Each interval in the union has length $\frac{2\delta(\varepsilon)}i$ and there  are $2i$ intervals. Hence, since $\cos$ is $1$-Lipschitz, for any $j\in\mathbb N$, 
	$\cos(2\pi j D)$ has Lebesgue measure at most $4\delta(\varepsilon)$. If $\varepsilon$ is small enough, we have $4\delta(\varepsilon)<2=\text{length}([-1,1])$. Hence in that case, as $\mathsf P(X_j\in[a,b])>0$ for any $0\le a<b\le 1$, \begin{equation*}\mathsf P(X_j\in[-1,1]\setminus\cos(2\pi j D)\mid X_i>1-\varepsilon)=0<\mathsf P(X_j\in[-1,1]\setminus\cos(2\pi j D))\end{equation*} and thus the pair $(X_i, X_j)$ is not independent.
\end{proof}
\end{example**}

\begin{example**}\label{example:Gaussians}
Let $Z_n, n\in\N$, be independent standard normal random variables (i.e. each $Z_n$ has a normal distribution with mean 0 and variance 1). Let $W$ be a random variable, independent of every $Z_n$ for $n\in\N$, such that $\mathsf P(W=0)=\mathsf P(W=1)=\sfrac12$. Then the statements \ref{enum:Uncorrelated}, \ref{enum:Dependent}, \ref{enum:SLLN} from Example \ref{example:Cosine} also hold for the random variables $X_n\define W Z_n$, $n\in\N$.

\begin{proof}
	To prove \ref{enum:Uncorrelated}, note that, by symmetry of each $Z_n$ and by the independence assumptions, for all $i\neq j\in\N$, 
	\begin{equation*}
		\mathsf E(X_i)\mathsf E(X_j) = \mathsf E(W)\mathsf E(Z_i)\mathsf E(X_j)=0\quad\text{and}\quad\mathsf E(X_i X_j)=\mathsf E(W^2)\mathsf E(Z_i)\mathsf E(Z_j) = 0.
	\end{equation*}
	To prove \ref{enum:Dependent}, notice that if there exists any $i\in\N$ for which $X_i>0$, then $W=1$. Hence, for all $i,j\in\N$ with $i\neq j$,
	\begin{equation*}
		\mathsf P(X_j>0\mid X_i>0)=\mathsf P(Z_j>0)=\frac12>\mathsf P(Z_j>0 \text{ and } W=1)=\frac14=\mathsf P(X_j>0).
	\end{equation*}
	Finally I will prove \ref{enum:SLLN}. Since $\frac{X_1+\dots+X_n}n=W\frac{Z_1+\dots+Z_n}n$, and $\frac{Z_1+\dots+Z_n}n$ has a normal distribution with mean 0 and variance $\frac1n$, we see that the SLLN holds for the sequence $(X_n)_{n\in\N}$ in question.
\end{proof}
\begin{remark}
	Example \ref{example:Gaussians} also is an example of a sequence of random variables for which the SLLN holds, but for which the non-negativity/bounded-from-below assumption in Theorem \ref{Thm:KolmogoroffBedingung} is not satisfied. If we try to treat this problem by applying Theorem \ref{Thm:KolmogoroffBedingung} separately to $X_n^+\define\max\{X_n, 0\}$ and $X_n^-\define-\min\{X_n,0\}$, $n\in\N$, then we will run into other issues because
	\begin{equation}\label{eq:Variance of the sum}
		\mathsf V(X_1^++\dots+X_n^+)=n\, \mathsf V(X_1^+)+\binom n2 (\mathsf E(X_1^+ X_2^+)-\mathsf E(X_1^+)\mathsf E(X_2^+)),
	\end{equation}
	while
	\begin{equation}\label{eq:Sum of the variances}
		\mathsf V(X_1^+)+\dots+\mathsf V(X_n^+)=n\mathsf V(X_1^+).
	\end{equation}
	We can compute\footnote{%
		See for instance \url{https://mathworld.wolfram.com/Erfc.html} for more on the integral of the complementary error function (which I used in the computation of $\int_0^\infty\mathsf P(Z_1>t)\,\mathrm dt$). In the second computation, I use that for $t>0$,
		\begin{equation*}
		\mathsf P(X_1^+ X_2^+>t)=\mathsf P(X_1 X_2>t \text{ and } X_1>0 \text{ and } X_2>0)=\frac 12\mathsf P(Z_1 Z_2>t\text{ and }Z_1>0\text{ and }Z_2>0).
		\end{equation*}%
	}
	\begin{align*}
		\mathsf E(X_2^+)
		=\mathsf E(X_1^+)
		=\int_0^\infty\mathsf P(X_1^+>t)\,\mathrm dt
		=\frac12\int_0^\infty\mathsf P(Z_1>t)\,\mathrm dt
		=\frac1{2\sqrt{2\pi}}
		&\approx 0.20
	\shortintertext{and}
		\mathsf E(X_1^+ X_2^+)
		=\int_0^\infty \mathsf P(X_1^+ X_2^+>t)\,\mathrm dt=\int_0^\infty\int_0^\infty\int_{t/x}^\infty \frac{\exp\left(-\frac{x^2}2\right)}{\sqrt{2\pi}}\frac{\exp\left(-\frac{y^2}2\right)}{\sqrt{2\pi}}\,\mathrm dy\,\mathrm dx\,\mathrm dt
		&\approx 0.16.
	\end{align*}
	In particular, we see that $\mathsf E(X_1^+ X_2^+)>\mathsf E(X_1^+)\mathsf E(X_2^+)$, and hence by \eqref{eq:Variance of the sum}, $\mathsf V(X_1^++\dots+X_n^+)$ is bigger than some positive constant (independent of $n$) times $\binom n2$ for all $n\in\N$. Comparing this with \eqref{eq:Sum of the variances} shows that the variables $X_n^+, n\in\N$ are not even quasi-uncorrelated.
\end{remark}
\end{example**}

\begin{remark}
	These two examples were added thanks to a suggestion by Jordan Stoyanov to provide examples of (infinite) sequences of random variables that are pairwise uncorrelated but not pairwise independent. The above examples can be considered as an extension to infinite sequences of a property for a pair of random variables. (Examples of such pairs can be found in \cite{Stoyanov}, chapter 7).
\end{remark}

\section{Proofs} 
\subsection{Tools for Theorem \ref{Thm:KolmogoroffBedingung}}
\begin{definition}\label{Def:m_n s_n}
	From now on, fix real numbers $\alpha > 1$ and $\varepsilon > 0$. The random variables $X_n$ are as in Theorem \ref{Thm:KolmogoroffBedingung}. Recall that, by assumption, $A\define\sup_{n\in\N} \frac{\Erwartung S_n}n<\infty$. Hence it is possible to define the integer $L \define \floor*{\frac{A}{\varepsilon}}$, where $\floor{\bullet}$ denotes \enquote{the largest integer smaller than, or equal to, \textbullet} (also known as the \enquote{floor part}).
	
	\smallskip
	For every $n\in\mathbb N$ let $m(n) \define \floor{\log_\alpha n}$. Then 
	\begin{gather}	\nonumber
		m(n) \to \infty \text{ as }n\to\infty\quad\text{and}\quad\alpha^{m(n)} \le n < \alpha^{m(n)+1} \text{ for all }n.
	\intertext{For a non-negative random variable $X_n$, let $s(n)\in\set{0, \dots, L}$ be natural numbers such that}
		\label{Abschaetzung S_n}
		\frac{\Erwartung(S_n)}n \in[\varepsilon\cdot s(n), \varepsilon\cdot(s(n)+1)[.
	\end{gather}
	The numbers $s(n)$ are always uniquely well-defined because the intervals in \eqref{Abschaetzung S_n} are disjoint and their union contains the interval $[0, A]$.
\end{definition}

\begin{definition}\label{Def:k_n}
	Let the random variables $X_n, n\in\N$, be as in Theorem \ref{Thm:KolmogoroffBedingung}. For any $n\in\mathbb N$ and $s \in\set{0,\dots, L}$, consider the sets 
	\begin{equation}\label{eq:Definition der Menge T}
		T_{n,s} \coloneqq \set{k\in\mathbb N: \alpha^{n} 
		\le k < \alpha^{n+1} \text{ and } \varepsilon\cdot s \le 
		\frac{\Erwartung(S_k)}k < \varepsilon\cdot(s+1)}.
	\end{equation}

	For some $n$ and $s$, $T_{n,s}$ can be empty. However, $T_{m(n),s(n)}$ is never empty as $n\in T_{m(n),s(n)}$ for any $n\in\mathbb N$. If $T_{n,s}$ is not empty, set $k(n,s)^+ \define \max T_{n,s}$ and $k(n,s)^- \define \min T_{n,s}$. Otherwise, if $T_{n,s}$ is empty, take \begin{equation*}k(n,s)^+=k(n,s)^-\define\floor{\alpha^n}.\end{equation*}	
	From now on, I will use the symbol $k(n,s)^\pm$ in statements that are true
	for both $k(n,s)^+$ and $k(n,s)^-$, respectively.
	For example, we have $k(n,s)^\pm \geq \floor{\alpha^n}\to \infty$ as $n\to\infty$. 
\end{definition}

\begin{remark}
	If we are dealing with $\frac{S_n-\mathsf ES_n}{b_n}$ instead of $\frac{S_n-\mathsf ES_n}n$ as described in Remark \ref{rmk:Arbitrary sequence}, then one has to replace $\frac{\mathsf ES_n}n$ in \eqref{Abschaetzung S_n} and \eqref{eq:Definition der Menge T} with $\frac{\mathsf ES_n}{b_n}$.
\end{remark}
The next statement is about the Kolmogoroff condition for the $S_{k(n,s)^\pm}$ subsequences.

\begin{lemma}\label{lemma:Subsequences}
	Let $\Folge Xn$ be a sequence of random variables satisfying the assumptions of Theorem 1. Then the $S_{k(n,s)^\pm}$ subsequences satisfy the Kolmogoroff condition, i.e.
	\begin{equation}\label{KolmogoroffBedingung1}
		\sum_{n=1}^\infty \frac{\Varianz (S_{k(n,s)^\pm})}{(k(n,s)^\pm)^2} < \infty \text{ for any } s \in \set{0, \dots, L}.
	\end{equation} 
\end{lemma}

\begin{remark}
	Lemma \ref{lemma:Subsequences} is not surprising because $(X_n)_{n\in\N}$ satisfies the Kolmogoroff condition \eqref{eq:KolmogoroffBedingung} and the numbers $k(n,s)^\pm$ grow exponentially in $n$.
\end{remark}
\begin{remark}
	If we consider the case of Remark \ref{rmk:Arbitrary sequence}, then in the same way one can prove that
	\begin{equation*}
		\sum_{n=1}^\infty \frac{\mathsf VS_{k(n,s)^\pm}}{b_{k(n,s)^\pm}^2}<\infty.
	\end{equation*}
\end{remark}
	
\begin{proof}[Proof of Lemma \ref{lemma:Subsequences}] Because the variables $X_n$ are quasi-uncorrelated, we have
	\begin{equation}\label{lemma 1 Teil1}
		\begin{split}
			\sum_{n=1}^\infty \frac{\Varianz S_{k(n,s)^\pm}}{(k(n,s)^\pm)^2} 
			&\le c \sum_{n=1}^\infty\bigg(\frac{1}{(k(n,s)^\pm)^2} \sum_{j=1}^{k(n,s)^\pm} \Varianz X_j \bigg) 
			\\
			&=c\sum_{j=1}^\infty \bigg(\Varianz X_j \sum_{\set{n: k(n,s)^\pm \geq j}}\frac{1}{(k(n,s)^\pm)^2}\bigg)
	\end{split}
	\end{equation}
	(the last equality holds by rearranging the terms).
	For brevity, I define $\kappa_j\define\sum_{\set{n: k(n,s)^\pm \geq j}}\frac{1}{(k(n,s)^\pm)^2}$.
	By Definition \ref{Def:k_n}, we have that for all $n,s$,
	\begin{equation*}
		\floor{\alpha^n} \le k(n,s)^\pm\le \alpha^{n+1} \text{ so that } \set{n: k(n,s)^\pm \geq j} \subset \set*{n: \alpha^{n} \geq \frac j\alpha}.
	\end{equation*}
	It follows that
	\begin{equation}\label{eq:Kappabound}
			\kappa_j \le \sum_{\set*{n:\alpha^{n} \geq \frac j\alpha}} \frac{1}{\floor{\alpha^n}}
			= \frac{\alpha^2}{j^2} \sum_{n=0}^\infty \frac{1}{\floor{\alpha^n}^2}
			\le \frac{\alpha^2}{j^2 } C \frac{\alpha^2}{\alpha^2-1} \text{ for some } C>0.
	\end{equation}
	In the last inequality, I have used that $\sum_{n=0}^\infty \frac{1}{\floor{\alpha^n}^2}$ behaves like a geometric series.
	Using \eqref{eq:Kappabound} in \eqref{lemma 1 Teil1} we get, as the $X_n$ satisfy the Kolmogoroff condition,
	\begin{equation*}
		\sum_{n=1}^\infty \frac{\Varianz S_{k(n,s)^\pm}}{(k(n,s)^\pm)^2} 
		\le c\cdot C\frac{\alpha^4}{\alpha^2-1} \sum_{j=1}^\infty \frac{\Varianz X_j}{j^2} < \infty.
	\end{equation*}
	 This completes the proof of \eqref{KolmogoroffBedingung1}.
\end{proof}

\begin{lemma}[SLLN for the $S_{k(n,s)^\pm}$ subsequences]\label{lem:SLLN for the subsequences}
	Let $\Folge Xn$ be a sequence of random variables as in Theorem 1. Then the SLLN holds for the subsequences $S_{k(n,s)^\pm}, n\in\N$, i.e. we have 
	\begin{equation*}
		\lim_{n\to\infty} \frac{S_{k(n,s)^\pm} - \Erwartung S_{k(n,s)^\pm}}{k(n,s)^\pm} = 0 \text{ a.s. for any } s \in\set{0, \dots, L}.
	\end{equation*}
\end{lemma}
\begin{remark}
	If we deal with the case of Remark \ref{rmk:Arbitrary sequence}, then one can prove in the same way as below that
	\begin{equation*}
		\lim_{n\to\infty} \frac{S_{k(n,s)^\pm} - \Erwartung S_{k(n,s)^\pm}}{b_{k(n,s)^\pm}} = 0 \text{ a.s. for all } s \in\set{0, \dots, L}.
	\end{equation*}
\end{remark}
\begin{proof}[Proof of Lemma \ref{lem:SLLN for the subsequences}]
	Fix $s\in\set{0, \dots, L}$. For $\delta>0$, define the events \begin{equation*}A_\delta(n) \define \set*{\abs*{\frac{S_{k(n,s)^\pm}}{k(n,s)^\pm} - \Erwartung{\frac{S_{k(n,s)^\pm}}{k(n,s)^\pm}}} > \delta}.\end{equation*}
	Then
	\begin{gather}
		\begin{split}
			A_\delta(n) &= \set*{\abs*{S_{k(n,s)^\pm} - \Erwartung{S_{k(n,s)^\pm}}} > k(n,s)^\pm \cdot 
			\delta} \\&= \set*{\abs*{S_{k(n,s)^\pm} - \Erwartung{S_{k(n,s)^\pm}}}^2 > (k(n,s)^\pm)^2 \cdot \delta^2}.
		\end{split}
		\shortintertext{We can now use the Chebyshev inequality and Lemma 1 to derive that}\label{Bounded sum}
		\sum_{n=1}^\infty \mathsf P(A_\delta(n)) \overset{\text{Chebyshev}}\le
		\frac{1}{\delta^2} \sum_{n=1}^\infty \frac{\Varianz S_{k(n,s)^\pm}}{(k(n,s)^\pm)^2}
		\overset{\text{Lemma 1}}< \infty.
	\end{gather}
	Consider the sets
	\begin{equation*}
		A_\delta \coloneqq \limsup_{n\to\infty} A_\delta(n).
	\end{equation*}
	The relation \eqref{Bounded sum} allows us to apply the Borel-Cantelli Lemma to obtain that $\mathsf P(A_\delta) = 0$ for any $\delta > 0$. This implies the a.s. convergence of 	$\abs*{\frac{S_{k(n,s)^\pm}}{k(n,s)^\pm} - \Erwartung{\frac{S_{k(n,s)^\pm}}{k(n,s)^\pm}}}$ to $0$ as $n\to\infty$ by standard results from probability theory and thus also proves Lemma 2.
\end{proof}

\subsection{Proof of Theorem \ref{Thm:KolmogoroffBedingung}}
\begin{proof}[Proof of Theorem \ref{Thm:KolmogoroffBedingung}.]
	From now on, I will abbreviate $k(m(n),s(n))^\pm$ by $k^\pm$. The idea of the proof is to use the fact that the SLLN holds for the $k^\pm$-subsequence and conclude that it also holds for the sequence $(X_n)_{n\in\N}$ using monotonicity arguments. This is the place where the non-negativity of the $X_n$'s is necessary.
	
	\smallskip
	Using definitions \ref{Def:m_n s_n}, \ref{Def:k_n} and the bound \eqref{Abschaetzung S_n}, we obtain that for any $\varepsilon>0$ and $n \in\mathbb N$: 
	\begin{equation}\label{k_n}
		k^- \le n \le k^+ \quad\text{ and }\quad
		\abs*{\frac{\Erwartung S_{k^\pm}}{k^\pm} - \frac{\Erwartung S_n}{n}} \le \varepsilon.
	\end{equation}
	Since all $X_n$'s are non-negative, we have
	\begin{equation}\label{S_n}
		0 \le S_n \le S_{n'} \text{ for all } n  \le n'.
		\text{ In particular, } 
		0 \le \Erwartung S_n \le \Erwartung S_{n'}.
	\end{equation}
	Note that by Definitions \ref{Def:m_n s_n} and \ref{Def:k_n}, we have
	$n\geq k^+ / \alpha$. Thus $1/n\le \alpha/k^+$ and 
	\begin{gather}
		\left(\frac{1}{n}-\frac{\alpha}{k^+}\right) S_{k^+}
		+ (\alpha-1)\frac{\Erwartung S_{k^+}}{k^+}
		\le (\alpha-1) B, \label{eq:Upper bound (alpha-1) times B}
	\shortintertext{since $\frac{\Erwartung S_{k^+}}{k^+}\le B$ and $\left(\frac{1}{n}-\frac{\alpha}{k^+}\right) S_{k^+}\le 0$. Bound \eqref{eq:Upper bound (alpha-1) times B} is equivalent to}
	\label{letzte Ungleichung}
		\frac{S_{k^+}}{n} - \frac{\Erwartung S_{k^+}}{k^+} \le 
		\frac{\alpha}{k^+} (S_{k^+} - \Erwartung S_{k^+}) + (\alpha - 1)B.
	\end{gather}
	It is also true that
	$\alpha k^- \geq n$. Thus $1/(\alpha k^-) \le 1/n$. By the definition of $B$,
	\begin{equation}\label{erste Ungleichung}
	\begin{split}
		-\left(1-\frac1\alpha\right)B + \frac{1}{\alpha k^-}(S_{k^-} - \Erwartung S_{k^-}) 
		&\le -\left(1-\frac1\alpha\right)\frac{1}{k^-}\Erwartung S_{k^-} + \frac{1}{\alpha k^-}(S_{k^-}-\Erwartung S_{k^-}) \\
		&= \frac{S_{k^-}}{\alpha k^-} - \Erwartung{} \frac{S_{k^-}}{k^-}
		\underset{\eqref{k_n}}\le \frac{S_{k^-}}{n} - \Erwartung{} \frac{S_n}{n} + \varepsilon
	\end{split}
	\end{equation}
	where $\frac{S_{k^-}}{\alpha k^-}\le\frac{S_{k^-}}n$ was used in the last inequality.
	In conclusion, referring to \eqref{erste Ungleichung}, \eqref{S_n} and \eqref{letzte Ungleichung},
	\begin{equation}\label{Grosse Gleichung}%
		 \begin{split}
		 -\varepsilon-\left(1-\frac1\alpha\right) B + \frac{1}{\alpha k^-} (S_{k^-} - 
		 \Erwartung S_{k^-}) 
		 &\le \frac{S_{k^-}}{n} - \frac{\Erwartung S_n}{n} \\
		 &\le \frac{1}{n}(S_n - \Erwartung S_n) \\
		 &\le \frac{1}{n} S_{k^+} -
		 \frac{1}{k^+} \Erwartung S_{k^+} + \varepsilon \\
		 &\le \frac{\alpha}{k^+} (S_{k^+} - \Erwartung S_{k^+}) + (\alpha - 1)B + \varepsilon.
		 \end{split}
	\end{equation}
	Lemma 2 guarantees that for any $\alpha > 1$ and  $\varepsilon > 0$, there exists a measurable event $\Omega_{\alpha, \varepsilon}\subset\Omega$ such that 
	\begin{equation*}\mathsf P(\Omega_{\alpha, \varepsilon}) = 1\quad\text{ and }\quad\lim_{n\to\infty} \dfrac{S_{k^\pm}(\omega) - \Erwartung S_{k^\pm}}{\alpha k^\pm} = 0 \text{ for all $\omega \in \Omega_{\alpha, \varepsilon}$}.
	\end{equation*}
	
	\smallskip
	Using the previously proven inequality \eqref{Grosse Gleichung} yields
	\begin{equation}\label{Ergebnis}
		 -\varepsilon - \left(1-\frac 1\alpha\right)B
		 \le \liminf_{n\to\infty} \frac{1}{n}(S_n(\omega) - \Erwartung S_n)
		 \le \limsup_{n\to\infty} \frac{1}{n}(S_n(\omega) - \Erwartung S_n)
		 \le (\alpha - 1)B + \varepsilon
	\end{equation}
	for every $\omega \in \Omega_{\alpha,\varepsilon}$. This holds for any
	$\alpha > 1$ and $\varepsilon > 0$.
	By $\sigma$-additivity of $\mathsf P$, we have 
	\begin{equation}\label{sigma}
		\Prob{}\Big(\bigcap_{k,m\in\mathbb N} \Omega_{1+\frac{1}{m},\frac{1}{k}}\Big) = 1.
	\end{equation}
	If we substitute $\alpha = 1 + \frac 1m$ and $\varepsilon = \frac 1k$ for $m,k\in\N$ into \eqref{Ergebnis} and let $m$ and $k$ go to $\infty$ simultaneously, then both \enquote{outer sides} of \eqref{Ergebnis} go to $0$. It follows from \eqref{sigma} that 
	\begin{equation*}\lim_{n\to\infty} \frac{1}{n}\big(S_n(\omega) - \Erwartung S_n\big) = 0\end{equation*} for almost every $\omega \in \Omega$.
	
	\medskip
	This completes the proof of Theorem \ref{Thm:KolmogoroffBedingung}.
\end{proof}
\begin{remark}
	Note that \eqref{Ergebnis} becomes
	\begin{equation*}
		-\varepsilon - \left(1-\frac 1\alpha\right)B 
		\le\liminf_{n\to\infty} \frac{1}{b_n}(S_n(\omega) - \Erwartung S_n)
		\le \limsup_{n\to\infty} \frac{1}{b_n}(S_n(\omega) - \Erwartung S_n)
		\le (\alpha - 1)B + \varepsilon
	\end{equation*}
	in the case described in Remark \ref{rmk:Arbitrary sequence}.
\end{remark}

\subsection{Proofs of Corollary \ref{Cor:PaarweiseUnabhKolmogoroff} and Corollary \ref{Cor:ZuTheorem2}}
\begin{proof}[Proof of Corollary \ref{Cor:PaarweiseUnabhKolmogoroff}]
	We can assume without loss of generality that the variables $X_n, n\in\N$ are centered, i.e., that $\Erwartung X_n = 0$ for all $n\in\N$ (otherwise consider $\tilde X_n\define X_n-\Erwartung X_n$). Since the variables $X_n, n\in\N$ are pairwise independent, we see that the sequences $(X_n^+)_{n\in\N}$ and $(X_n^-)_{n\in\N}$, consist of non-negative and pairwise independent random variables and that each sequences satisfies the Kolmogoroff condition (since $\Varianz X_n^+\le \Varianz X_n$ and $\Varianz X_n^-\le \Varianz X_n$),\footnote{For any $n\in\N$, one has $\mathsf V(X_n)=\mathsf E(X_n^2)-\mathsf E(X_n)^2=\mathsf E(X_n^2)\ge\mathsf E\Big((X_n^+)^2\Big)\ge\mathsf V\Big(X_n^+\Big)$. The same is true for $X_n^-$. In the first inequality, I have used that $(X_n^+)^2=\max\{X_n,0\}^2\le X_n^2$.} and that
	\begin{equation*}
		\sup_{n\in\N} \frac{\sum_{k=1}^n\Erwartung{} X_k^+}n\le\sup_{n\in\N}\frac{\sum_{k=1}^n\Erwartung{} \abs{X_k}}n<\infty.
	\end{equation*}
	The analogous relation is true for $(X_k^-)_{k\in\N}$.
	
	\smallskip
	Hence, by Theorem \ref{Thm:KolmogoroffBedingung}, we get
	\begin{subequations}
	\begin{align}\label{eq:Conclusion for X plus}
		\lim_{n\to\infty} \frac{\sum_{k=1}^n X_k^+-\Erwartung X_k^+}n=0, \\ \label{eq:Conclusion for X minus}
		\lim_{n\to\infty} \frac{\sum_{k=1}^n X_k^--\Erwartung X_k^-}n=0.
	\end{align}
	\end{subequations}
	Adding the two proves the Corollary.
\end{proof}

\begin{proof}[Proof of Corollary \ref{Cor:ZuTheorem2}]
	By using the same decomposition as before and then directly applying Theorem \ref{Thm:PaarweiseUnabh} to $(X_n^+)_{n\in\N}$ and $(X_n^-)_{n\in\N}$, we get \eqref{eq:Conclusion for X plus} and \eqref{eq:Conclusion for X minus}, so that we can conclude once again by adding.
\end{proof}

\subsection{Tools for Theorem \ref{Thm:PaarweiseUnabh}}
I start with a statement showing that the expected value of the truncations does not differ too much from the original.
\begin{lemma}\label{lem:ErwartungswerteSindUngefaehrGleich}
	Let $\Folge Xn$ be a sequence of non-negative random variables such that, as $n\to\infty$, $X_n-Y_n\to0$ in $L^1$-sense, where $Y_n=X_n\cdot 1_{\set{X_n\le n}}$. Then, for $T_n=Y_1+\dots+Y_n$,
	\begin{equation*}
		\lim_{n\to\infty} \frac{\Erwartung{S_n} - \Erwartung{T_n}}n = 0.
	\end{equation*}
\end{lemma}

\begin{proof}
	The proof is straight-forward. Indeed,
	\begin{equation}\label{eq:AbschaetzungDerDifferenzDerErwartungswerte}
		0\le\lim_{n\to\infty} \frac{\Erwartung{S_n} - \Erwartung{T_n}}n
		= \lim_{n\to\infty} \frac{\sum_{k=1}^n \Erwartung{}\left(X_k -Y_k\right)}n=0,
	\end{equation}
	where the last equality is true by Cesàro summation.
\end{proof}

The next result shows that if the SLLN holds for $(T_n)_{n\in\N}$, then it also holds for $(S_n)_{n\in\N}$.

\begin{lemma}\label{T_n reicht}
	Let $\Folge Xn$ be a sequence of non-negative random variables such that, as $n\to\infty$, $X_n-Y_n\xrightarrow{\text{a.s.}}0$. Using the same notation as in Lemma \ref{lem:ErwartungswerteSindUngefaehrGleich}, we have
	\begin{equation*}
		\lim_{n\to\infty} \frac{S_n-T_n}n = 0 \text{ a.s.}
	\end{equation*}
\end{lemma}
\begin{proof}
	Same idea as in Lemma \ref{lem:ErwartungswerteSindUngefaehrGleich},
	\begin{equation*}
		0\le\frac{S_n-T_n}n=\frac{\sum_{k=1}^n X_k-Y_k}n \xrightarrow{\text{a.s.}} 0 \text{ as }n\to\infty,
	\end{equation*}
	by Cesàro-summation.
\end{proof}

\subsection{Proof of Theorem \ref{Thm:PaarweiseUnabh}}
\begin{proof}[Proof of Theorem \ref{Thm:PaarweiseUnabh}.]
	From Lemmas \ref{lem:ErwartungswerteSindUngefaehrGleich} and \ref{T_n reicht}, we know that for almost every $\omega\in\Omega$, we have
	\begin{equation*}
	\begin{split}
		\lim_{n\to\infty} \frac{S_n(\omega)-\Erwartung S_n}n &= \lim_{n\to\infty} \frac{T_n(\omega)-\Erwartung T_n}n+\frac{\Erwartung T_n-\Erwartung S_n}n+\frac{S_n(\omega)-T_n(\omega)}n\\&=\lim_{n\to\infty} \frac{T_n(\omega)-\Erwartung T_n}n.
	\end{split}
	\end{equation*}
	Since the sequence $(Y_n)_{n\in\N}$ satisfies the Kolmogoroff condition by assumption and since the variables $Y_n,n\in\N$ are pairwise independent, we can apply Theorem \ref{Thm:KolmogoroffBedingung} to obtain that
	\begin{equation*}
		\lim_{n\to\infty} \frac{T_n-\Erwartung T_n}n=0 \text{ a.s.}
	\end{equation*}
	This completes the proof.
\end{proof}

\subsection{Proof of Theorem \ref{Prop:Identical Distribution}}

\begin{proof}[Proof of Theorem \ref{Prop:Identical Distribution}.]
	I will first prove that the sequence $(Y_n)_{n\in\N}$ satisfies the Kolmogoroff condition. For this, let us bound the tail sums of the convergent series $\sum_{j=1}^\infty \frac 1{j^2}$. Recall that $\sum_{j=1}^\infty \frac 1{j^2}=\frac{\pi^2}6$, which is the famous Basel problem. I claim that
	\begin{equation}\label{eq:BaselBounds}
		\sum_{j=k}^\infty \frac1{j^2}\le\frac{1}k\cdot\frac{\pi^2}6 \text{ for any }k\in\N.
	\end{equation}
	Notice that if $k=1$, then \eqref{eq:BaselBounds} is true (we have equality). For $k\geq 2$, we have
	\begin{equation*}
		\sum_{j=k}^\infty \frac1{j^2} < \int_{k-1}^\infty \frac1{x^2}\,\mathrm dx = \frac1{k-1}\le\frac2k<\frac{1}k\cdot\frac{\pi^2}6.
	\end{equation*}
	
	Now, using \eqref{eq:BaselBounds} and the abbreviation $g_k\define 1_{\set{k<X_1\le k+1}}$,
	\begin{equation*}
	\begin{split}
		\sum_{j=1}^\infty \frac{\Varianz Y_j}{j^2}&\le\sum_{j=1}^\infty \frac{\Erwartung{}(Y_j^2)}{j^2} 
		=\sum_{j=1}^\infty \frac{\Erwartung(X_1^2 \cdot 1_{\set{X_1 \le j}})}{j^2} \\
		&=\lim_{N\to\infty} \sum_{j=1}^N \left(\frac{1}{j^2} \sum_{k=0}^{j-1} \Erwartung(X_1^2 \cdot g_k)\right)\\
		& = \lim_{N\to\infty} \sum_{k=0}^{N-1} \left(\Erwartung (X_1^2 \cdot g_k) \sum_{j=k+1}^N \frac{1}{j^2}\right)\\
		& 
		\le\frac{\pi^2}6\lim_{N\to\infty} \sum_{k=0}^{N-1} \frac{\Erwartung(X_1^2 \cdot g_k)}{k+1}\\
		&= \frac{\pi^2}6\sum_{k=0}^{\infty} \frac{1}{k+1} \int_{\{k<X_1 \le k+1 \}} X_1^2 \, \mathrm d \mathsf P\\
		&\le\frac{\pi^2}6\sum_{k=0}^\infty \frac{k+1}{k+1} \Erwartung{(X_1 \cdot g_k)}
		= \frac{\pi^2}6\Erwartung X_1 < \infty.
	\end{split}
	\end{equation*}
	
	It remains to show that $X_n-Y_n\to0$ as $n\to\infty$ in $L^1$-sense and a.s. First, I will prove the $L^1$ convergence. We have $\Erwartung{}\abs{X_n-Y_n}=\Erwartung(X_n-Y_n)=\Erwartung(X_n\cdot1_{\set{X_n>n}})=\Erwartung(X_1\cdot1_{\set{X_1>n}})$. Since $\Erwartung X_1<\infty$ by assumption, $\Erwartung(X_1\cdot1_{\set{X_1>n}})$ will converge to $0$, as $n\to\infty$, by Lebesgue's dominated convergence Theorem.
	
	Now, regarding the a.s. convergence, we have 
	\begin{equation*}
		\sum_{n=1}^\infty \Prob(X_n\neq Y_n)=\sum_{n=1}^\infty \Prob (X_n>n)=\sum_{n=1}^\infty \Prob (X_1>n)\le\int_0^\infty \Prob(X_1\geq x)\,\mathrm dx=\Erwartung X_1<\infty.
	\end{equation*}
	Hence, by Borel-Cantelli, with probability $1$ it is true that $X_n=Y_n$ for $n$ large enough. 
\end{proof}

\subsection*{Acknowledgements}
I would like to thank Jordan Stoyanov for his interest in my work and for his many useful comments and suggestions, most of which are taken into account in this version of the paper.


\begin{thebibliography}{9}
\bibitem{Etemadi} 
\textsc{Nasrollah Etemadi}, \textit{An elementary proof of the Strong Law of Large Numbers}. Zeitschrift für Wahrscheinlichkeitstheorie und Verwandte Gebiete 55, pages 119–122 (1981). \url{https://doi.org/10.1007/BF01013465}

\bibitem{Korchevsky}
\textsc{Valery Korchevsky}, \textit{On the Strong Law of Large Numbers for Sequences of Pairwise Independent Random Variables}. J Math Sci 244, pages 805–810 (2020). \url{https://doi.org/10.1007/s10958-020-04654-y}

\bibitem{Chandra-Goswami}
\textsc{Tapas K. Chandra} and \textsc{A. Goswami}, \textit{Cesáro Uniform Integrability and the Strong Law of Large Numbers.} Sankhyā: The Indian Journal of Statistics, Series A (1961-2002), vol. 54, no. 2, 1992, pages 215–231. \url{https://www.jstor.org/stable/25050875}

\bibitem{Shiryaev}
\textsc{Albert N. Shiryaev}, \textit{Probability 2}. Third english edition, 2018. Part of the Springer Graduate Texts in Mathematics book series (GTM, volume 95). \url{https://doi.org/10.1007/978-0-387-72208-5}

\bibitem{Csoergo}
\textsc{Sándor Csörgő} and \textsc{K. Tandori} and \textsc{Vilmos Totik}, \textit{On the Strong Law of Large Numbers for pairwise independent random variables}. Acta Math Hung 42, pages 319–330 (1983). \url{https://doi.org/10.1007/BF01956779}

\bibitem{Stoyanov}
\textsc{Jordan M. Stoyanov}, \textit{Counterexamples in Probability}. Third edition, Dover, Mineola (NY), 2014. Part of \textit{Dover Books on Mathematics}. \url{https://www.worldcat.org/title/counterexamples-in-probability/oclc/958218408}


\end{thebibliography}
\end{document}